\documentclass[11pt,oneside,leqno]{amsart}
\usepackage{amsxtra}
\usepackage{amsopn}
\usepackage{amsmath,amsthm,amssymb}
\usepackage{color}
\usepackage{amscd}
\usepackage{amsfonts}
\usepackage{latexsym}
\usepackage{verbatim}
\usepackage{pb-diagram}
\usepackage{pifont}
\usepackage{graphicx}
\usepackage{array}
\newcommand{\cmark}{\ding{51}}

\newcommand{\om}{\omega}

\theoremstyle{plain}
\newtheorem{theorem}{Theorem}[section]
\newtheorem*{theorem*}{Theorem}

\newtheorem*{mt*}{Main Theorem}
\newcommand\g{{\mathfrak{g}}}
\newcommand\h{{\mathfrak{h}}}

\newcommand\m{{\mathfrak{m}}}

\begin{document}
\title[Conformally 
flat Pseudo-Riemannian Homogeneous Ricci Solitons 4-Spaces]{ Conformally 
flat Pseudo-Riemannian Homogeneous Ricci Solitons 4-Spaces }
\author{M. Chaichi and Y. Keshavarzi}
\date{}

\address{Department of Mathematics\\Payame noor University\\P.O. Box 19395-3697\\Tehran\\Iran.}
\email{Mohammad Chaichi: chaichi@pnu.ac.ir, Yadollah Keshavarzi: y.keshavarzi@pnu.ac.ir}

\begin{abstract}
We consider four-dimensional conformally flat homogeneous pseudo-Riemannian manifolds.
 According to forms (Segre types) of the Ricci operator, we provide a full classification of four-dimensional 
pseudo-Riemannian conformally flat homogeneous Ricci solitons.
\end{abstract}

\maketitle

\section{Introduction}
\label{}
 A natural generalization of an Einstein manifold is Ricci soliton, i.e.   
a pseudo Riemannian metric $g$ on a smooth manifold $M$, such that the equation
\begin{eqnarray}
\begin{array}{cccc}
\mathcal{L}_{X} g =\varsigma g-\varrho,
\end{array}
\end{eqnarray}
holds for some $\varsigma \in R$ and some smooth vector field $X$ on $M$, where $\varrho$
denotes the Ricci tensor of (M, g) and $\mathcal{L}_{X}$ is the usual Lie derivative.
 According
to whether $\varsigma > 0, \varsigma = 0$ or $\varsigma < 0$ a Ricci soliton $g$ is said to be shrinking, steady or expanding, respectively. A homogeneous Ricci soliton on a homogeneous space $M = G/H$ is a G-invariant metric $g$ for which the equation (1.1) holds and
an invariant Ricci soliton is a homogeneous apace, such that equation (1.1) is satisfied by an invariant vector field.
Indeed, the study of Ricci solitons homogeneous spaces is an interesting area
of research in pseudo-Riemannian geometry.
For example, evolution of homogeneous Ricci solitons under the bracket flow \cite{Lm}, algebraic solitons and the Alekseevskii Conjecture properties\cite{LM}, conformally flat Lorentzian gradient Ricci
solitons\cite{MB}, properties of algebraic Ricci Solitons of three-dimensional Lorentzian
Lie groups\cite{BA}, algebraic Ricci solitons \cite{Ba}. Non-K\"{a}hler examples of Ricci solitons are very hard to find yet (see \cite{DH}).
Let $(G, g)$ be a simply-connected completely solvable Lie
group equipped with a left-invariant metric, and $(\g,\langle,\rangle )$ be the corresponding
metric Lie algebra. Then $(G, g)$ is a Ricci soliton if and only if $(\g,\langle ,\rangle)$ is an
algebraic Ricci soliton \cite{LJ}. 

 A pseudo-Riemannian manifold $\mathcal{M}=(M,g)$ is said to be homogeneous if the group $G$ of isometries acts transitively on $\mathcal{M}$. In this case, $(M,g)$ can be written as $(G/H,g)$, where $H$ is the isotropy group at a fixed point $\circ$ of $M$ and $g$ is an invariant pseudo-Riemannian metric.
Homogeneous manifolds have been used in several modern research in
pseudo-Riemannian geometry, for example, Lorentzian spaces for which all
null geodesics are homogeneous became relevant in physics \cite{Me},\cite{F} which following this fact, several studies on $g.o.$ spaces (that is, spaces whose geodesics
are all homogeneous) have done in last years (see \cite{Ca},\cite{Ca1},\cite{Ca2}). Homogeneous solutions to the Ricci flow have long been studied by many authors
(e.g., \cite{JI1},\cite{JI2},\cite{KI},\cite{LO},\cite{PP},\cite{TP},\cite{PB}).

The complete local classification of four-dimensional homogeneous pseudo-
Riemannian manifolds with non-trivial isotropy obtained \cite{KO},\cite{ko}. 
Using that the classification of four-dimensional conformally flat
at homogeneous pseudo-Riemannian manifolds $M = G/H$ was obtaind in \cite{Cal}.

Conformally flat spaces are the subject of many investigations in Riemannian and pseudo-Riemannian geometry. A conformally flat (locally) homogeneous Riemannian manifold is (locally) symmetric \cite{Ta}, and so, as proved in \cite{Ry}, it admits a universal covering either a spaces form $\mathbb R^n$, $\mathbb S^n(k)$, $\mathbb H^n(-k)$, or one of the Riemannian products $\mathbb R \times \mathbb S^{n-1}(k)$, $\mathbb R\times \mathbb H^{n-1}(-k)$ and $\mathbb S^{p}(k)\times \mathbb H^{n-p}(-k)$.

Riemannian locally conformally flat complete shrinking and steady gradient Ricci
solitons were recently classified \cite{C1},\cite{C2}.
Conformally flat Einstein pseudo-Riemannian manifolds have constant sectional curvature.
By the way, they are symmetric. Conformally flat homogeneous Riemannian
manifolds are always symmetric\cite{Ta} . On the other hand, some of our examples show the
existence of conformally flat homogeneous pseudo-Riemannian Ricci solitons which are
not symmetric.
The purpose of this paper is to investigate four-dimensional conformally flat homogeneous Ricci
solitons pseudo-Riemannian manifolds by focusing on the Segre types of their Ricci operator. First we give a details of local classification of four-dimensional conformally flat homogeneous pseudo-Riemannian manifolds obtained in \cite{Cal}, then following that we study and classify four-dimensional conformally flat homogeneous Ricci
solitons.

This paper is organized as follows. In Section 2, we recall 
some basic facts on Ricci
solitons , which play important roles in studying homogeneous Ricci solitons.
In Section 3, we report some necessary results on conformally flat homogeneous pseudo-Riemannian manifolds obtained in \cite{Cal}. In Section 4  we shall investigate several geometric properties of four-dimensional conformally flat
pseudo-Riemannian homogeneous Ricci solitons. 

 \section{perliminaries}

Let $M = G/H$ be a homogeneous manifold (with $H$ connected), $\g$ the Lie algebra
 of $G$ and $\h$ the isotropy subalgebra. Consider $\m =\g /\h$ the factor space, which identies
with a subspace of $\g$ complementary to $\h$. The pair $(\g,\h)$ uniquely defines the isotropy
representation
\begin{center}
$\psi :\g \longrightarrow \mathfrak{gl}(\m),\quad      \psi(x)(y)=[x,y]_\m$,   
\end{center}
for all  $x\in\g, y\in\m$. Suppose that $\lbrace e_1,...,e_r,u_1,...,u_n\rbrace$ be a basis of $\g$, where $\lbrace e_j\rbrace$ and $\lbrace u_i\rbrace$ are bases of $\h$ and $\m$
respectively, then with respect to $\lbrace u_i \rbrace$, $H_j$ whould be the isotropy representation  for $e_j$.
A bilinear form $B$ on $\m$ is invariant if and only if $\psi(x)^t\circ B+B\circ \psi(x)=0$, for all $x \in \h$, where
$\psi(x)^t$ denotes the transpose of $\psi(x)$. In particular, requiring that $B=g$ is symmetric and
nondegenerate, this leads to the classification of all invariant pseudo-Riemannian metrics
on $G/H$.

 Following the notation given in \cite{Cal}, $g$ on $\m$ uniquely defines its invariant linear Levi-Civita connection, as the corresponding homomorphism of $\h$-modules  $\Lambda:\g \longrightarrow \mathfrak{gl}(\m)$ such that $\Lambda(x)(y_\m)=[x,y]_\m$ for all  $x\in \h, y\in \g$. In other word
\begin{eqnarray}
\Lambda(x)(y_\m)=\frac{1}{2}[x,y]_\m+v(x,y),
\end{eqnarray}
for all  $x,y\in \g$, where $v:\g \times \g\rightarrow \m$ is the $\h$-invariant symmetric mapping uniquely determined by
\begin{center}
$2g(v(x, y), z_\m) = g(x_\m, [z,y]_\m) + g(y_\m,[z,x]_\m)$,
\end{center}
for all  $x,y,z\in \g$,
Then the curvature tensor can be determined by
\begin{eqnarray}
R:\m \times \m \longrightarrow \mathfrak{gl}(\m),\quad  R(x,y)=[\Lambda(x),\Lambda(y)]-\Lambda ([x,y]),   
\end{eqnarray}
and with respect to $u_i$, the Ricci tensor $\varrho$ of $g$  is given by
\begin{eqnarray}
\varrho (u_i ,u_j)=\sum_{k=1}^4g(R(u_k,u_i)u_j,u_k),\quad i,j=1,\dots,4.
\end{eqnarray}
Furthermore, whenever $X = \sum_{k=1}^4x_k e_k $ the Equation (1.1) becomes
$$
\begin{array}{cr}
\sum_{k=1}^4x_k(g([u_k,u_i],u_j)+g(u_i,[u_k,u_j]))+\varrho(u_i,u_j)=\varsigma g_{ij}, \quad \quad i,j=1,...,4.
\end{array}
$$
Moreover, the Equation (2.2) characterizes conformally
flat pseudo-Riemannian manifolds of dimension
$n\geq4$, while it is trivially satisfied by any three-dimensional manifold. Conversely,
the condition
\begin{eqnarray}
\begin{array}{cr}
\nabla_{i} \varrho_{jk}-\nabla_{j} \varrho_{ik}=\frac{1}{2(n-2)}(g_{jk}\nabla_{i}\tau-g_{ik}\nabla_{j}\tau),
\end{array}
\end{eqnarray}
which characterizes three-dimensional conformally flat spaces, is trivially satisfied
by any conformally flat Riemannian manifold of dimension greater than three.

 Following the notation and the classifcation used in \cite{ko}, the space identified by the
type $n.m^k:q$ is the one corresponding to the $q$-th pair $(\g,\h)$ of type $n.m^k$, where $n={\rm dim}(\h)$
$(= 1,..., 6)$, $m$ is the number of the complex subalgebra $\h^{\mathbb C}$ of ${\mathfrak so}(4,\mathbb C)$ and $k$ is the number
of the real form of $\h^{\mathbb C}$. When the index $q$ is removed, we refer simultaneously to all
homogeneous spaces corresponding to pairs $(\g, \h)$ of type $n.m^k$.

We now recall the possible Segre types of the Ricci operator for a conformally flat homogeneous four-dimensional manifold through the following tables \cite{Cal}.
\begin{center}
{\small
\medskip
{\bf Table 1: Segre types of $Q$ for an inner product of signature $(2,2)$.}\nopagebreak \\[3 pt]
\begin{tabular}{|p{3.5cm}|p{2cm}|p{1.5cm}|p{1.5cm}|p{1.5cm}|p{1.5cm}|}
\hline
Case&Ia&Ib&Ic&IIa&IIb\\
\hline
Non degenerate type &---& $[1,11\bar 1]$& $[1,\bar {1}1\bar 1]$ &---& [22]\\
\hline
Degenerate type & $[(11),(11)]\newline$ $[(1|1,1|1)]\newline$ $[(11,1)1)]\newline$ $[1(1,11))]\newline$ $[(11,11)]$ & $[(1,1)1\bar 1]$& $[(1,\bar {1}1\bar 1)]$ & $[(1,1)2]\newline$ $[1,(12)]\newline$ $[(1,12)]\newline$& [(22)]\\
\hline
Case & IId & IIc & IIIa & IIIb & IV\\
\hline
Non degenerate type & $[21\bar 1]$ & $[2\bar 2]$ & [13] & [1,3] & [4]\\
\hline
Degenerate types &---&---& [(13)] & [(1,3)] &---\\
\hline
\end{tabular}
}\\

{\small
\medskip
{\bf Table 2: Segre types of $Q$ for a Lorentzian inner product.}\nopagebreak \\[3 pt]
\begin{tabular}{|p{3.5cm}|p{2cm}|p{1.5cm}|p{1.5cm}|p{1.5cm}|}
\hline
Case & Ia & Ib & II & III \\
\hline
Non degenerate type &---& $[11,1\bar 1]$& --- & [1,3]\\
\hline
Degenerate type & $[(11),(1,1)]\newline$ $[1(11,1)]\newline$ $[(111)1)]\newline$ $[(111,1)]$ &  $[(11),1\bar 1]$ & $[(11),2]\newline$ $[1(1,2)]\newline$ $[(11,2)]$& [(1,3)]\\
\hline
\end{tabular}
}\\
\end{center}

\begin{theorem}\cite{Ho}\label{CFdiag}
Let $M_q^n$ be an $n(\geq 3)$-dimensional conformally flat homogeneous pseudo-Riemannian manifold with diagonalizable Ricci operator. Then, $M_q^n$ is locally isometric to one of the following:
\begin{itemize}
\item[(i)] A pseudo-Riemannian space form;
\item[(ii)] A product manifold of an $m$-dimensional space form of constant curvature $k \neq 0$ and an $(n-m)$-dimensional pseudo-Riemannian manifold of constant curvature $-k$, where $2\leq m\leq n-2$;
\item[(iii)] A product manifold of an $(n-1)$-dimensional pseudo-Riemannian manifold of index $q - 1$ of constant curvature $k\neq 0$ and a one-dimensional Lorentzian manifold, or a product of an $(n-1)$-dimensional pseudo-Riemannian manifold of index $q$ of constant curvature $k \neq 0$ and a one-dimensional Riemannian
manifold.
\end{itemize}
\end{theorem}
It is obvious from the last theorem that if $(M,g)$ have digonalizable Ricci operator then the Ricci operator is degenerate. So the study of cases with nondegenerate Ricci operator restricts to the non-diagonalizable ones.

\section{Cases with nondegenerate Ricci operator}

As we mentioned befor if $(M,g)$ have digonalizable Ricci operator then the Ricci operator is degenerate \cite{Cal}. So let $(M,g)$ be a conformally flat homogeneous four dimensional manifold with nondegenerate Ricci operator.
For any point $p\in M$, we have that $g(0, p)=\{ 0\}$  if and only if $ Q_{p} $ is
nondegenerate. Therefore, $(M,g)$ is locally isometric to a Lie group equipped with a left-invariant pseudo-Riemannian metric and the Ricci operator of conformally flat homogeneous pseudo-Riemannian four-manifolds can only be of Segre type $[1,11\bar{1}]$ if $g$ is neutral, or $[11,1\bar{1}]$ if $g$ is Lorentzian \cite{Cal}. We report here the Lie group structure of the mentioned types and their Ricci tensor as follow.
\begin{theorem}\cite{Cal}\label{ND-NU}
Let $(M,g)$ be a conformally flat homogeneous four-dimensional manifold with the Ricci operator of Segre type $[1,11\bar1]$. Then, $(M,g)$ is locally isometric to {one of the unsolvable Lie groups $SU(2) \times \mathbb R$ or $SL(2,\mathbb R) \times \mathbb R$}, equipped with a left-invariant neutral metric, admitting a pseudo-orthonormal basis  $\{e_1,e_2,e_3,e_4\}$ for their Lie algebra, 
such that the Lie brackets take one of the following forms:
$$
\begin{array}{llll}
{\rm i)} \quad & [e_1,e_2]=\varepsilon\alpha e_3,\quad &[e_1,e_3]=-\varepsilon\alpha e_2,\quad & [e_2,e_3]=2\alpha(e_1+\varepsilon e_4),\\[2pt] &
[e_2,e_4]=-\alpha e_3,\quad & [e_3,e_4]=\alpha e_2, &
\end{array}
$$
$$
\begin{array}{llll}
{\rm ii)} \quad & [e_1,e_2]=-\varepsilon\alpha e_1,\quad &[e_1,e_3]=\alpha e_1,\quad & [e_1,e_4]=2\alpha(\varepsilon e_2-e_3),\\[2pt] &
[e_2,e_4]=-\varepsilon\alpha e_4,\quad & [e_3,e_4]=\alpha e_4,&
\end{array}
$$
\end{theorem}
and the Ricci tensor in case (i) is given by
 $$
\varrho= \left( \begin{array}{cccc}
   -2\alpha^2+2\epsilon^2\alpha^2 & 0 & 0 & 4\epsilon \alpha^2  \\
   0 & 2\alpha^2+4\epsilon \alpha^2-2\epsilon^2\alpha^2 & 0 & 0  \\
   0 & 0 & 4\epsilon \alpha^2-2 \alpha^2+2 \epsilon^2 \alpha^2 & 0  \\
   4\epsilon \alpha^2 & 0 & 0 &2\alpha^2-2\epsilon^2\alpha^2
 \end{array}  \right),
$$
also the Ricci tensor in case (ii) is then given by
 \begin{eqnarray}
\varrho= \left( \begin{array}{cccc}
  2\epsilon^2\alpha^2-2\alpha^2 & 0 & 0 & -2\epsilon^2\alpha^2-2\alpha^2  \\
   0 & -4\epsilon^2\alpha^2 & 0 & 0  \\
   0 & 0 & -4\alpha^2 & 0  \\
   -2\epsilon^2\alpha^2-2\alpha^2 & 0 & 0 & -2\epsilon^2\alpha^2+2\alpha^2
 \end{array}  \right),
\end{eqnarray}
where $\alpha\neq 0$ is a real constant and $\varepsilon=\pm1$.

\begin{theorem}\cite{Cal}\label{ND-LO} Let $(M,g)$ be a conformally flat homogeneous Lorentzian four-manifold with the Ricci operator of Segre type $[11,1\bar1]$. Then, $(M,g)$ is locally isometric to one of the unsolvable Lie groups $SU(2) \times \mathbb R$ or $SL(2,\mathbb R) \times \mathbb R$, equipped with a left invariant Lorentzian metric, admitting a pseudo-orthonormal basis  $\{e_1,e_2,e_3,e_4\}$ for the Lie algebra, such that the Lie brackets take one of the following forms:
$$
\begin{array}{llll}
{\rm i)} \quad & [e_1,e_2]=-2\alpha(\varepsilon e_3+e_4),\quad&[e_1,e_3]=\varepsilon\alpha e_2,\quad& [e_1,e_4]=\alpha e_2,\\[2pt] &
[e_2,e_3]=\varepsilon\alpha e_1,\quad & [e_2,e_4]=\alpha e_1,
\end{array}
$$
$$
\begin{array}{llll}
{\rm ii)} \quad & [e_1,e_2]=2\alpha(\varepsilon e_3+e_4),\quad&[e_1,e_3]=\varepsilon\alpha e_2,\quad& [e_1,e_4]=\alpha e_2,\\[2pt]
& [e_2,e_3]=\varepsilon\alpha e_1,\quad & [e_2,e_4]=\alpha e_1,
\end{array}
$$
\end{theorem}
and the Ricci tensor in case (i) is given by
 $$
\varrho= \left( \begin{array}{cccc}
  4\alpha^2 & 0 & 0 & 0  \\
   0 & -4\epsilon^2\alpha^2 & 0 & 0  \\
   0 & 0 & 0 &-4\alpha^2 \epsilon  \\
  0 & 0 & -4\alpha^2\epsilon & 0
 \end{array}  \right),
$$
also the Ricci tensor in case (ii) is then given by
 $$
\varrho= \left( \begin{array}{cccc}
 - 4\alpha^2  \epsilon^2 & 0 & 0 & 0  \\
   0 & 4\alpha^2 & 0 & 0  \\
   0 & 0 & 0 &-4\alpha^2 \epsilon  \\
  0 & 0 & -4\alpha^2\epsilon & 0
 \end{array}  \right),
$$
where $\alpha\neq 0$ is a real constant and $\varepsilon=\pm1$.

Now using the above classification statements we classify conformally flat homogeneous Ricci soliton four dimensional manifolds with nondegenerate Ricci operator. The result is the following theorem.
\begin{theorem}\label{nondegWalker}
Let $(M,g)$ be a conformally flat homogeneous four dimensional manifold with nondegenerate Ricci operator. Then $(M,g)$ can not be a Ricci soliton manifold.
\end{theorem}
\begin{proof}
According to the above argument for signature $(2,2)$ and Lorentzian signature, we have the explicit description of Lie groups and their Lie algebras. We report the
calculations for the case $(ii)$ of signature $(2,2)$ with $\varepsilon=\pm1$. 
Using (2.2) to compute $\Lambda_i:=\Lambda(e_i)$ for all indices i = 1, .., 4, so describing the
Levi-Civita connection of g. We get
$$
 \Lambda_1= \left( \begin{array}{cccc}
   0 &-\epsilon \alpha & \alpha & 0  \\
    \epsilon \alpha & 0 &0 & \epsilon \alpha  \\
 \alpha & 0 & 0 & -\alpha  \\
   0 & \epsilon \alpha & \alpha & 0
 \end{array}  \right),\quad\quad
  \Lambda_2= \left( \begin{array}{cccc}
   0 & 0 & 0 &\epsilon \alpha  \\
   0& 0 &0 &0   \\
   0 &0 & 0 & 0  \\
   \epsilon \alpha & 0 & 0 & 0
 \end{array}  \right),
$$
$$ 
 \Lambda_3= \left( \begin{array}{cccc}
   0 & 0 & 0 & \alpha  \\
   0& 0 &0 &0   \\
   0 &0 & 0 & 0  \\
    \alpha & 0 & 0 & 0
 \end{array}  \right),\quad\quad
  \Lambda_4= \left( \begin{array}{cccc}
  0 &\epsilon \alpha & \alpha & 0  \\
    -\epsilon \alpha & 0 &0 & \epsilon \alpha  \\
 \alpha & 0 & 0 & \alpha  \\
   0 & \epsilon \alpha & -\alpha & 0
 \end{array}  \right).
 $$
Ricci tensors can be now deduced from the above formulas by a direct
calculation applying (2.3). In particular, the Ricci tensor in this case has the form described
in (3.6) where $\alpha\neq 0$ is a real constant and $\varepsilon=\pm1$.

We prove the existence of a vector field which determine a Ricci soliton
 in any possible case leads to a contradiction. Choose the pseudo-orthonormal basis $\{e_1,e_2,e_3,e_4\}$, and an arbitrary vector field $X = \sum_{k=1}^4x_k e_k $ and a real constant $\varsigma$, by $(1.1)$ we find that $X$ and $\varsigma$ determine a Ricci soliton if and only if the
components $x_k$ of $X$ with respect to $\lbrace e_k\rbrace$ and $\varsigma$ satisfy\\
$$
 \left\lbrace \begin{array}{cc}
 \hspace*{-39mm} -4\alpha^2+\varsigma=0 &\\
\hspace*{-36mm} -4\epsilon^2\alpha^2-\varsigma=0&\\
\hspace*{-32mm} -2\epsilon^2\alpha^2-2\alpha^2=0&\\
\hspace*{-35mm}  x_1\alpha-2x_4\alpha=0&\\
 \hspace*{-35mm}     2x_1\alpha+x_4\alpha=0&\\[2pt]
 \hspace*{-30mm}    -x_1\epsilon \alpha-2x_4\epsilon \alpha=0&\\[2pt]
  \hspace*{-32mm}   2x_1 \epsilon \alpha-x_4 \epsilon \alpha=0&\\[2pt]
    2x_2 \epsilon \alpha-2x_3 \alpha-2\epsilon^2 \alpha^2+2 \alpha^2+\varsigma=0&\\[2pt]
     2x_2 \epsilon \alpha-2x_3 \alpha+2\epsilon^2\alpha^2-2\alpha^2-\varsigma=0 &
\end{array} \right.$$
From third equation we find that $\alpha=0$ which is imposible.
\end{proof}
\section{Cases with degenerate Ricci operator and trivial isotropy}
By the arguments of the previous section, now we proceed the manifolds with degenerate Ricci operator. For Ricci parallel examples by Proposition 4.1 of \cite{Cal} it must be noted that  the conformally flat Ricci parallel homogeneous Walker spaces are one of the spaces of the Theorem \ref{CFdiag}, or admit a two step nilpotent Ricci operator.
Now, let $(M,g)$ be a not Ricci parallel (and so not locally symmetric) conformally flat homogeneous manifold with degenerate Ricci operator. First, we proceed the cases with trivial isotropy. Separating the diagonalizable Ricci operator cases, such spaces are locally isometric to a Lie group $G$, equipped with a left-invariant neutral metric, and $Q$ has one of the Segre types: $[1,(12)]$, $[1,12]$, $[22]$, $[(13)]$ and $[(1,3)]$. Also, for the Lorentzian signature, $Q$ admits the Segre types either $[(11,2)]$, or $[(1,3)]$ (see \cite{Cal}).
\begin{theorem}\label{DegTriv1}
Let $(M,g)$ be a conformally flat not Ricci-parallel four-dimensional Lie group with the Ricci operator of Segre types  $[1,(12)]$, $[22]$, $[(13)]$ and $[(1,3)]$, then $(M,g)$ is not a Ricci soliton manifold.
\end{theorem}
\begin{proof}
We apply the same argument used to prove Theorem $3.1$, proving that in the case  $[(22)]$  there is not any Ricci soliton. Consider the pseudo-orthonormal basis $\{e_1,e_2,e_3,e_4\}$, such that
$$
\begin{array}{cccc}
\quad&[e_1,e_2]=\frac{1-4k_1^2}{4k_1}e_2+\frac{1}{8k_1}e_4,\quad &[e_1,e_3]=\frac{1+8k_1^2}{4k_1}e_1+\frac{1+8k_1^2}{4k_1}e_3,\\ [2pt]\quad& [e_1,e_4]=\frac{1+16k_1^2}{8k_1}e_2+\frac{1+4k_1^2}{4k_1}e_4,\quad&
 [e_2,e_3]=\frac{1+4k_1^2}{4k_1}e_2+\frac{1+16k_1^2}{8k_1}e_4,\\[2pt]\quad&
[e_3,e_4]=\frac{-1}{8k_1}e_2-\frac{1-4k_1^2}{4k_1}e_4,
\end{array}
$$
for any real constant $k_1$.
Using (2.2) to compute $\Lambda_i:=\Lambda(e_i)$ for all indices i = 1, .., 4, so describing the
Levi-Civita connection of g. We get
$$
 \Lambda_1= \left( \begin{array}{cccc}
   0 &0 & \frac{1+8k_1^2}{4k_1} & 0  \\
    0 & 0 &0 & \frac{1+8k_1^2}{8k_1}  \\
  \frac{1+8k_1^2}{4k_1} & 0 & 0 & 0  \\
   0 & 0 & \frac{1+8k_1^2}{8k_1} & 0
 \end{array}  \right),
  \Lambda_2= \left( \begin{array}{cccc}
   0 & -\frac{-1+4k_1^2}{4k_1} & 0 & k_1  \\
   \frac{-1+4k_1^2}{4k_1}& 0 &\frac{1+4k_1^2}{4k_1} &0   \\
   0 &\frac{1+4k_1^2}{4k_1} & 0 & -k_1  \\
   k_1 & 0 & k_1 & 0
 \end{array}  \right),
$$
$$ 
 \Lambda_3= \left( \begin{array}{cccc}
  0 &0 & -\frac{1+8k_1^2}{4k_1} & 0  \\
    0 & 0 & 0 & -\frac{1+8k_1^2}{8k_1}  \\
  -\frac{1+8k_1^2}{4k_1} & 0 & 0 & 0  \\
   0 & 0 & -\frac{1+8k_1^2}{8k_1} & 0
 \end{array}  \right),
  \Lambda_4= \left( \begin{array}{cccc}
  0 & k_1 & 0 & -\frac{1+4k_1^2}{4k_1}  \\
   -k_1 & 0 & -k_1 & 0 \\
   0 & -k_1 & 0 & \frac{-1+4k_1^2}{4k_1}  \\
   -\frac{1+4k_1^2}{4k_1} &0 & \frac{-1+4k_1^2}{4k_1} & 0
 \end{array}  \right).
 $$
By a direct calculation the curvature and Ricci tensors can be now obtaind from formulas (2.3) and (2.4) . In particular, in this case the Ricci tensor has the form 
 $$
\varrho= \left( \begin{array}{cccc}
 1 & 0 & -1 & 0  \\
   0 & 1 & 0 & -1  \\
   -1 & 0 & 1 & 0  \\
  0 & -1 & 0 & 1
 \end{array}  \right).
$$
For an arbitrary vector field $X = \sum_{k=1}^4x_k e_k $ and a real constant $\varsigma$, by $(1.1)$ we find that $X$ and $\varsigma$ determine a Ricci soliton if and only if the
components $x_k$ of $X$ with respect to $\lbrace e_k\rbrace$ and $\varsigma$ satisfy
\begin{eqnarray}
 \hspace*{0mm} \frac{x_2}{8k_1}+\frac{x_4(1+4k_1^2)}{4k_1}=0\\
 \hspace*{0mm} -\frac{x_4(1+16k_1^2)}{8k_1}+\frac{x_2(-1+4k_1^2)}{4k_1}=0\\
 \hspace*{0mm} \frac{x_4}{8k_1}+\frac{x_2(1+4k_1^2)}{4k_1}=0\\
  \hspace*{0mm} -\frac{x_2(1+16k_1^2)}{8k_1}+\frac{x_4(-1+4k_1^2)}{4k_1}=0\\
 \hspace*{0mm}  1+\varsigma-\frac{x_1(1+8k_1^2)}{2k_1} =0\\
     \hspace*{0mm} \frac{ x_1(1+8k_1^2)}{4k_1}+\frac{x_3(1+8k_1^2)}{4k_1}=1
 \end{eqnarray}
\begin{eqnarray}
   \hspace*{0mm} -\frac{x_3(1+8k_1^2)}{2k_1}+1-\varsigma=0\\
         \hspace*{0mm} -\frac{x_1(-1+4k_1^2)}{2k_1}+1-\varsigma+\frac{x_3(1+4k_1^2)}{2k_1}=0\\
       \hspace*{0mm} -\frac{x_1(1+4k_1^2)}{2k_1}+1+\varsigma+\frac{x_3(-1+4k_1^2)}{2k_1}=0\\
 x_1 (-\frac{1}{8k_1}+\frac{(1+16k_1^2)}{8k_1})+x_3 (-\frac{1}{8k_1}+\frac{(1+16k_1^2)}{8k_1})=1
  \end{eqnarray}
The equations (4.8) and (4.9) together yield that $x_2=x_4=0$, also if we plus equations (4.11) and (4.13) together then we have $\frac{(x_1+x_3)(1+8k_1^2)}{2k_1} =2$. On the other hand, adding    (4.14) to (4.15) yield that $4k_1(x_1+x_3)=2$. Finally, from these two equations it is clear that $1=0$ which means that the equations system is not compatible.  
\end{proof}
Let $(M,g)$ be a conformally flat not Ricci-parallel four dimensional Lie group with the Ricci operator of Segre type $[(1,12)]$. Then, $(M,g)$ is locally isometric to the solvable Lie group $G=\mathbb R \ltimes \mathbb R^3$, equipped with a left invariant neutral metric, admitting a pseudo-orthonormal basis  $\{e_1,e_2,e_3,e_4\}$ for the Lie algebra, such that the 
  Lie algebra $\g$ is described by
$$\begin{array}{ll}
[e_1,e_2]=- [e_1,e_3]=-\frac{1}{2k_1}e_1-k_2e_2-k_2e_3, \quad &  [e_2,e_3]= \frac{2k_1^2+1}{2k_1}e_2+\frac{2k_1^2+1}{2k_1}e_3, \\[4 pt]
[e_2,e_4]=-[e_3,e_4]=k_3e_2+k_3e_3+k_1e_4,
\end{array}$$
for any real constants $k_1 \neq 0, k_2,k_3$.
\begin{theorem}\label{DegTriv1}
Let $(M,g)$ be a conformally flat not Ricci-parallel four-dimensional Lie group with the Ricci operator of Segre type $[1,(12)]$, then $(M,g)$ is a Ricci soliton manifold and
 this case occurs when
  $\varsigma$ is arbitrary and
 $$ 
\begin{array}{cccc}
X=\frac{k_1}{1+2k_1^2} e_2+\frac{k_1}{1+2k_1^2} e_3.
\end{array}
$$
\end{theorem}
\begin{proof}
We apply (2.2) to compute $\Lambda_i:=\Lambda(e_i)$ for all indices i = 1, .., 4, so describing the
Levi-Civita connection of $g$. We get
 $$
 \Lambda_1= \left( \begin{array}{cccc}
   0 &0 & -\frac{1}{2k_1} & \frac{1}{2k_1}  \\
    \frac{1}{2k_1} & 0 &0 & 0  \\
  \frac{1}{2k_1} & 0 & 0 & 0  \\
   0 & 0 & 0 & 0
 \end{array}  \right),\quad \quad
  \Lambda_3= \left( \begin{array}{cccc}
   0 & k_2 & -k_2 & 0  \\
   -k_2 & 0 & -\frac{1+2k_1^2}{2k_1} & -k_3  \\
   -k_2 & -\frac{1+2k_1^2}{2k_1} & 0 & -k_3  \\
   0 & -k_3 & k_3 & 0
 \end{array}  \right),
$$
$$ 
 \Lambda_2= \left( \begin{array}{cccc}
   0 & -k_2 & k_2 & 0  \\
   k_2 & 0 & \frac{1+2k_1^2}{2k_1} & k_3  \\
   k_2 & \frac{1+2k_1^2}{2k_1} & 0 & k_3  \\
   0 & k_3 & -k_3 & 0
 \end{array}  \right),\quad\quad\quad
  \Lambda_4= \left( \begin{array}{cccc}
  0 & 0 & 0 & 0  \\
   0 & 0 & 0 & -k_1  \\
   0 & 0 & 0 & -k_1  \\
   0 &-k_1 & k_1 & 0
 \end{array}  \right).
 $$
Ricci tensors can be now deduced from the above formulas by a direct
calculation applying (2.4). In particular, the Ricci tensor has the form 
 \begin{eqnarray}
\varrho= \left( \begin{array}{cccc}
 0 & 0 & 0 & 0  \\
   0 & 1 & -1 & 0  \\
   0 & -1 & 1 & 0  \\
  0 & 0 & 0 & 0
 \end{array}  \right).
\end{eqnarray}
 Let  $X = \sum_{k=1}^4x_k e_k $ be an arbitrary vector field and $\varsigma$ a real constant , using  all the needed information above, we obtain that the Ricci soliton condition
(1.1) is satisfied if and only if
\begin{eqnarray}
 \hspace*{0mm} -\frac{x_1}{2k_1}+x_2k_2-x_3k_2=0\\
 \hspace*{0mm} \frac{x_2}{k_1}-\frac{x_3}{k_1}-\varsigma=0\\
 \hspace*{0mm} -2x_2k_1+2x_3k_1+\varsigma=0\\
 \hspace*{0mm} -x_2k_3+x_3k_3-x_4k_1=0\\
 \hspace*{0mm}  -2x_1k_2-2x_4k_3+1+\varsigma-\frac{x_2(1+2k_1^2)}{k_1} =0\\
     \hspace*{0mm}  -2x_1k_2-2x_4k_3+1-\varsigma-\frac{x_3(1+2k_1^2)}{k_1}=0 \\
        \hspace*{0mm}  2x_1k_2+2x_4k_3-1+\frac{x_2(1+2k_1^2)}{2k_1}+\frac{x_3(1+2k_1^2)}{2k_1} =0
      \end{eqnarray}
     Since $k_1$ is an arbitrary real number, from equations (4.19) and (4.20) we get that $\varsigma=0$. Also equations (4.22) and (4.23) yield that $x_2=x_3$. Now with these results, by (4.18), $x_1=0$ and by (4.21) $x_4=0$, also again by equations (4.22) and (4.23) we get that 
     $x_2=x_3=\frac{k_1}{(1+2k_1^2)}$. So, $X$ has the form
     $$ 
\begin{array}{cccc}
\frac{k_1}{1+2k_1^2} e_2+\frac{k_1}{1+2k_1^2} e_3.
\end{array}
$$
Since $k_1\neq 0$, no Einstein cases
occur. 
\end{proof}

\section{Cases with degenerate Ricci operator and non-trivial isotropy}
In this section we consider cases conformally flat homogeneous, not locally symmetric pseudo-Riemannian four manifold with non-trivial isotropy . For these spaces, the approach is based on the classification of four dimensional homogeneous spaces with non-trivial isotropy presented by Komrakov in \cite{ko}. In \cite{Cal}, the authors checked case by case the Komrakov's list for conformally flat not Ricci parallel (and so not locally symmetric) examples with degenerate and not diagonalizable Ricci operator.

By using the lists which are presented in \cite{Cal} for the conformally flat non-symmetric homogeneous 4-spaces with non-trivial isotropy and non-diagonalizable degenerate Ricci operator,  among them, we are able to determine some different examples
of homogeneous spaces $M = G/H$ for which equation (1.1) holds for some vector fields
$X\in \m$ and some invariant metrics which are not Einstein.

 If $(M,g)$ be a conformally flat homogeneous, not locally symmetric pseudo-Riemannian four manifold, which its Ricci operator $Q$ is degenerate and not diagonalizale, Then the possible Segre types of $Q$ is either $[22]$,$[1,12]$ or $[11,2]$ (see \cite{Cal}).
  We can now state the following classification result.
\begin{theorem}
Among conformally flat homogeneous non-locally symmetric four-dimensional pseudo-Riemannian 
non-trivial Ricci soliton with the Ricci operator of Segre types $[22]$,$[1,12]$ or $[11,2]$, the Ricci Solitons
examples are listed in the following Tables 3, 4 and  5, where the checkmark means that $X$ is invariant for all Lie algebras of that form. 
\end{theorem}
\begin{proof}
As an example, we report here the calculations for the case $\bf{1.3^1.5}$. Let $M = G/H$ be a four-dimensional homogeneous
space, such that the isotropy subalgebra $\h$ is determined
by conditions (5.25) and the Lie algebra $\g$ has been determined befor. We apply (2.2) to compute $\Lambda_i:=\Lambda(e_i)$ for all indices i = 1, .., 4, so describing the
Levi-Civita connection of g. We get
$$
 \Lambda_1= \left( \begin{array}{cccc}
   0 & 0 & \frac{1}{2}\lambda & 0  \\
   0& 0 &0 &\frac{1}{2}\lambda   \\
   0 &0 & 0 & 0  \\
   0& 0 &0 &0 
 \end{array}  \right),\quad\quad
  \Lambda_2= \left( \begin{array}{cccc}
   0 & 0 & 1 &0  \\
   0& 0 &0 & 1   \\
   0 &0 & 0 & 0  \\
   0 & 0 & 0 & 0
 \end{array}  \right),
$$
$$ 
 \Lambda_3= \left( \begin{array}{cccc}
   \frac{1}{2}\lambda & 0 & \frac{c(2+\lambda^2)}{a\lambda} & \frac{c}{a}  \\
   1+\lambda^2& -\lambda & -\frac{c+c\lambda^2+b\lambda}{a} & \frac{c\lambda}{2a}   \\
   0 &0 & \lambda & 0  \\
    0 & 0 & 1+\lambda^2 & -\frac{1}{2}\lambda
 \end{array}  \right),\quad\quad
  \Lambda_4= \left( \begin{array}{cccc}
  -1 &0 & \frac{c}{a} & -\frac{2c}{a\lambda}  \\
   -\frac{1}{2}\lambda & 0 &\frac{c\lambda}{2a} & -\frac{c}{a}  \\
 0 & 0 & 0 & 0  \\
   0 & 0 & -\frac{1}{2}\lambda & 1
 \end{array}  \right).
 $$
 The curvature and Ricci tensors can be now deduced from the above formulas by a direct
calculation applying (2.2) and (2.3). Also the Ricci tensor has the form described
in (5.26). Now
Choose the pseudo-orthonormal basis ${u_1,u_2,u_3,u_4}$, and an arbitrary vector field $X = \sum_{k=1}^4x_k u_k \in \m$ and a real constant $\varsigma$, by $(1.1)$ we find that $X$ and $\varsigma$ determine a Ricci soliton if and only if the
components $x_k$ of $X$ with respect to $\lbrace u_k\rbrace$ and $\varsigma$ satisfy\\
$$
 \left\lbrace \begin{array}{cc}
\quad  a x_3=0 &\\
 -2x_1a+\frac{2\varsigma c}{\lambda}=0&\\
   x_4a+\varsigma a=0&\\
      x_1\varsigma a-\varsigma c=0&\\[2pt]
     -x_3 a(1+\lambda^2)-a\lambda x_4=0&\\[2pt]
     -a\lambda x_3-a x_4-\varsigma a=0&\\[2pt]
      -2a(1+\lambda^2)x_1+2a\lambda x_2+\frac{1}{2}\lambda^2-\varsigma b=-2&
      \end{array} \right.$$
From first equation and taking account $a\neq 0$ we find that $x_3=0$. Now by fifth equation $x_4$ must be zero and following that by third equation $\varsigma=0$. It is clear that $x_1=0$, so, 
$$ 
\begin{array}{cccc}
x_2=\frac{\lambda^2+4}{4a\lambda}.
\end{array}
$$
therefore $X$ has the form 
 $$ 
\begin{array}{cccc}
X=\frac{\lambda^2+4}{4a\lambda} u_2.
\end{array}
$$
 Since $x_2\neq 0$, no Einstein cases
occur. Finally, again by (5.25) we see at once that $X\in \m$ is invariant if and only if
$X \in Span\lbrace u_2\rbrace$. Therefore, $X$ 
is invariant and so, determines a homogeneous Ricci soliton.
\end{proof}
$\vspace*{3cm}$\\
{\small
\medskip
{\bf Table 3: Non-symmetric examples with ${\bf Q}$ of Segre type ${\bf [(22)]}$.}\nopagebreak \\[3 pt]
\begin{tabular}{!{\vrule width 1pt}p{1.1cm}|p{6.5cm}|p{3cm}|p{1cm}|p{2cm}!{\vrule width 1pt}}
\noalign{\hrule height 1pt}
Case&Invariant metric $g$&$X$&$\varsigma$&$X$ is invariant\\
\hline
$1.3^1$:$5$&$2a(-\om_1\om_4+\om_2\om_3)+\frac{2c\lambda\mu-d\lambda^2-\mu d-2c\lambda}{\mu(\mu-1)}\om_3\om_3\newline+2c\om_3\om_4+d\om_4\om_4$&$\frac{\mu(\mu-2)}{4a} u_1-\frac{\lambda \mu}{4a} u_2$&0 &\cmark\\
\hline
$1.3^1$:$28$&$2a(-\om_1\om_4+\om_2\om_3)+2\frac{1}{8\varsigma}\om_3\om_3+\frac{1}{8\varsigma}\om_4\om_4$&$x_1 u_1-\frac{7}{8a} u_2+\varsigma u_3$& $\neq 0$ & $\Leftrightarrow$ $x_1=0$\\
\hline
$1.3^1$:$29$&$2a(-\om_1\om_4+\om_2\om_3)+2\frac{1}{8\varsigma}\om_3\om_3-\frac{1}{8\varsigma}\om_4\om_4$&$x_1 u_1-\frac{7}{8a} u_2+\varsigma u_3$&  $\neq 0$ & $\Leftrightarrow$ $x_1=0$\\
\hline
$1.3^1$:$30$&$\begin{array}{c}$
$(1) \hspace*{2mm}  2a(-\om_1\om_4+\om_2\om_3)+b(\lambda^2-
\lambda)\om_3\om_3\\
-(d\lambda-d)\om_3\om_4
+d\om_4\om_4\\
(2) \hspace*{2mm} 2a(-\om_1\om_4+\om_2\om_3)$ $+b\om_3\om_4
+d\om_4\om_4\\
(3) \hspace*{2mm} 2a(-\om_1\om_4+\om_2\om_3)+b(\mu^2+
\mu)\om_3\om_3\\-(b\mu-d\mu-d-b)\om_3\om_4
+d\om_4\om_4$ $\end{array}$
&$\begin{array}{c}$
$\hspace*{-2mm}\frac{\lambda^2-1}{4a\lambda} u_2\\
\\
\hspace*{-2mm}x_1 u_1+\frac{b}{4ad} u_2-\frac{1}{2d} u_3\\
\hspace*{-2mm} \frac{-1+\mu^2}{4a}(u_1-u_2)\\
  \end{array}$
&$\begin{array}{c}
0\\
\\
 -\frac{1}{2d} \\ 0\\
\end{array}$&$\begin{array}{c}
$ $\hspace*{-7mm}$\cmark$\\$ $
\\
 \hspace*{-2mm}\Leftrightarrow$ $x_1=0 \\ \hspace*{-7mm}$ \cmark$\\$ $
\end{array}$\\
\noalign{\hrule height 1pt}
\end{tabular}
\newpage
}
{\small\medskip
{\bf Table 4: Non-symmetric examples with $\bf{Q}$ of Segre type $\bf{[(1,12)]}$.} \nopagebreak\\[3 pt]
\begin{tabular}{!{\vrule width 1pt}p{1cm}|p{5.5cm}|p{3.8cm}|p{1.3cm}|p{1.3cm}!{\vrule width 1pt}}
\noalign{\hrule height 1pt}
Case&Invariant metric $g$&$X$&$\varsigma$&$X$ is invariant\\
\hline
$1.1^1$:$1$ &$2a\om_1\om_3+2c\om_2\om_4+d\om_4\om_4$ &$\begin{array}{c}$ $\hspace*{-4mm}(1)\hspace*{1mm}\frac{-a^2+c^2+2\varsigma d a^2}{4a^2 c} u_2-\varsigma u_4\\ \hspace*{-19mm} (2)\hspace*{1mm} \frac{\varsigma d}{2a} u_2-\varsigma u_4$ $\end{array}$ & $\hspace*{-1mm}[-\infty,\infty]$& $\begin{array}{c}
$ $\hspace*{0mm}$\cmark$\\$ $  \hspace*{0mm}$ \cmark $
\end{array}$\\ \hline
$1.1^1$:$2$ &$\begin{array}{c}$ $\hspace*{-4mm}(1)\hspace*{1mm} 2a\om_1\om_3+2c\om_2\om_4+d\om_4\om_4,\hspace*{1mm} \lambda=0\\ 
\hspace*{-13mm} (2)\hspace*{1mm}2a\om_1\om_3+2c\om_2\om_4+d\om_4\om_4$ $\end{array}$& $\begin{array}{c}$ $\hspace*{-4mm}(1)\hspace*{1mm}x_1 u_1+\frac{1+2\varsigma d}{4 c} u_2+\varsigma u_4\\ \hspace*{-11mm} (2)\hspace*{1mm}x_1 u_1+\frac{-1+2\lambda }{4 c\lambda} u_2$ $\end{array}$ &$\begin{array}{c}
$ $\hspace*{-4mm}[-\infty,\infty]\\$ $  \hspace*{0mm}0$  $
\end{array}$& $\begin{array}{c}
$ $\hspace*{-4mm}\Leftrightarrow$ $x_1=0\\$ $  \hspace*{-4mm}$ \cmark $
\end{array}$\\
 \hline
$1.3^1$:$5$ &$2a(-\om_1\om_4+\om_2\om_3)+b\om_3\om_3+2c\om_3\om_4-\frac{2c}{\lambda}\om_4\om_4, \hspace*{12mm} \mu=0$ &$\frac{\lambda^2+4}{4a\lambda} u_2$&$  \hspace*{6mm}0$&$  \hspace*{4mm}$ \cmark \\
\hline
$1.3^1$:$7$&$2a(-\om_1\om_4+\om_2\om_3)+b\om_3\om_3+2c\om_3\om_4-2c\om_4\om_4$ &$\frac{-1-2\lambda \mu+\mu^2+\lambda^2}{4a} u_1$&$  \hspace*{6mm}0$&$  \hspace*{4mm}$ \cmark \\ \hline
$1.3^1$:$12$ &
$\begin{array}{c}
\hspace*{-7mm}(1)\hspace*{3mm} 2a(-\om_1\om_4+\om_2\om_3)+b\om_3  \om_3\\+2c\om_3\om_4+d\om_4\om_4\\
 \hspace*{-7mm} (2)\hspace*{3mm}2a(-\om_1\om_4+\om_2\om_3)+b\om_3  \om_3\\+2c\om_3\om_4+
d\om_4\om_4,\hspace*{1mm} \lambda=0,\hspace*{1mm} \mu=0\\$
$\hspace*{-7mm} (3)\hspace*{3mm}2a(-\om_1\om_4+\om_2\om_3)+b\om_3  \om_3\\+
d\om_4\om_4,\hspace*{1mm} \lambda=0\\$
$\hspace*{-7mm} (4)\hspace*{3mm}2a(-\om_1\om_4+\om_2\om_3)+b\om_3  \om_3\\+2c\om_3\om_4+
d\om_4\om_4,\hspace*{1mm} \lambda=1-\mu\\$
$\hspace*{-7mm}(5)\hspace*{3mm}2a(-\om_1\om_4+\om_2\om_3)+b\om_3  \om_3\\+2c\om_3\om_4+
d\om_4\om_4,,\hspace*{1mm} \mu=\frac{1}{2}\\$
$\hspace*{-7mm}(6)\hspace*{3mm}2a(-\om_1\om_4+\om_2\om_3)+b\om_3  \om_3\\+2c\om_3\om_4+
d\om_4\om_4,\hspace*{1mm} \lambda=0,\hspace*{1mm} \mu=\frac{1}{2}$
$\end{array}$& 
$\begin{array}{c}$ $\hspace*{-8mm}(1)\hspace*{1mm} \frac{-1-2\lambda \mu+\mu^2+\lambda^2}{4a} u_1\\
$ $\\$ $
\hspace*{-7mm} (2)\hspace*{1mm}\frac{-1+2\varsigma d}{4a}u_1+x_2 u_2\\
+\varsigma u_4\\$
$\hspace*{-2mm}(3)\hspace*{1mm}\frac{-1+\mu^2+2\varsigma d}{4a}u_1+x_2 u_2\\
+\varsigma u_4\\$
$\hspace*{-15mm}(4)\hspace*{1mm}\frac{\mu(-1+\mu^2)}{a} u_1\\$
$\\$
$\hspace*{-12mm}(5)\hspace*{1mm}\frac{-4\lambda-3+4\lambda^2}{16a} u_1\\$
$\\$
$\hspace*{-8mm}(6)\hspace*{1mm}\frac{3+8\varsigma d}{16a}u_1+x_2 u_2\\
-\varsigma u_4$
$\end{array}$ &
$\begin{array}{c}$ $\hspace*{-8mm}\hspace*{1mm} 0\\$
$ \\ $ $
\hspace*{-4mm} \hspace*{1mm}[-\infty,\infty]\\$
$\\$
$\hspace*{-4mm}\hspace*{1mm}[-\infty,\infty]\\
$ $\\$ $
\hspace*{-8mm}\hspace*{1mm}0\\$
$\\$
$\hspace*{-8mm}\hspace*{1mm}0\\$
$\\$
$\hspace*{-8mm}\hspace*{1mm}0$
$\end{array}$ 
&$\begin{array}{c}$ $\hspace*{-6mm}\hspace*{1mm}$  \cmark $ \\$
$ \\ 
\hspace*{-4mm} \hspace*{1mm}\Leftrightarrow$ $x_2=0\\$
$\\$
$\hspace*{-5mm}\hspace*{1mm}\Leftrightarrow$ $x_2=0\\
$ $\\
\hspace*{-6mm}\hspace*{1mm}$ \cmark $\\$
$\\$
$\hspace*{-6mm}\hspace*{1mm}$  \cmark $\\$
$\\$
$\hspace*{-5mm}\hspace*{1mm}\Leftrightarrow$ $x_2=0$
$\end{array}$ 
 \\ \hline
$1.3^1$:$19$ &$2a(-\om_1\om_4+\om_2\om_3)+2c\om_3\om_4+d\om_4\om_4$&$
\frac{1}{4a} u_1$&$  \hspace*{6mm}0$&$  \hspace*{4mm}$ \cmark \\ \hline
{$1.3^1$:$21$}$\vphantom{\displaystyle\frac{A^a}{A^a}}$&$2a(-\om_1\om_4+\om_2\om_3)+2c\om_3\om_4+d\om_4\om_4$&
$\frac{(-2+\lambda)\lambda}{4a} u_1$&$  \hspace*{6mm}0$&$  \hspace*{4mm}$ \cmark \\ \hline
$1.3^1$:$30$ &
$\begin{array}{c}
\hspace*{-6mm} (1)\hspace*{3mm}2a(\om_2\om_3-\om_1\om_4)+b\om_3\om_3\\+{b(1-\mu)}\om_3\om_4+d\om_4\om_4,\hspace*{1mm} \lambda=1\\$
$\hspace*{-8mm} (2)\hspace*{3mm}2a(\om_2\om_3-\om_1\om_4)+b\om_3\om_3\\+{b(1-\mu)}\om_3\om_4-\frac{1}{2\varsigma}\om_4\om_4,\\ \hspace*{1mm} \lambda=1,\hspace*{1mm}\mu=0\\$
$\hspace*{-6mm}(3)\hspace*{3mm}2a(\om_2\om_3-\om_1\om_4)-\frac{1}{2\varsigma}\om_3\om_3\\+{b(1-\mu)}\om_3\om_4+d\om_4\om_4,\\  \hspace*{1mm} \lambda=0,\hspace*{1mm}\mu=1\\$
$\hspace*{-8mm}(4)\hspace*{3mm}2a(\om_2\om_3-\om_1\om_4)+b\om_3\om_3\\+{b(1-\mu)}\om_3\om_4+d\om_4\om_4,\hspace*{1mm} \mu=1$
$\end{array}$& 
$\begin{array}{c}$ $\hspace*{-18mm}(1)\hspace*{1mm}\frac{-\mu^2+1}{4a\mu} u_1 \\
$ $\\
\hspace*{-1mm} (2)\hspace*{1mm} x_1 u_1+\frac{b}{4ad}u_2+\varsigma u_3\\$ $\\$ $\\$
$\hspace*{-3mm}(3)\hspace*{1mm}\frac{\varsigma d}{2a}u_1+x_2 u_2+\varsigma u_4\\$ $\\$ $\\$
$\hspace*{-20mm}(4)\hspace*{1mm}\frac{\lambda^2-1}{4a\lambda} u_2\\$
$\end{array}$ &
$\begin{array}{c}$ $\hspace*{-3mm}\hspace*{1mm} 0\\$
$ \\ $ $
\hspace*{-4mm} \hspace*{1mm}[-\infty,\infty]\\$
$\\$ $\\$
$\hspace*{-4mm}\hspace*{1mm}[-\infty,\infty]\\
$ $\\$ $\\
\hspace*{-3mm}\hspace*{1mm}0\\$
$\end{array}$ 
&$\begin{array}{c}$ $\hspace*{-6mm}\hspace*{1mm}$  \cmark $ \\$
$ \\ 
\hspace*{-4mm} \hspace*{1mm}\Leftrightarrow$ $x_1=0\\$
$\\$ $\\$
$\hspace*{-5mm}\hspace*{1mm}\Leftrightarrow$ $x_2=0\\
$ $\\$ $\\
\hspace*{-6mm}\hspace*{1mm}$ \cmark $\\$
$\end{array}$ 
 \\
\hline
$1.4^1$:$10$ &$a(-2\om_1\om_3+\om_2\om_2)+b\om_3\om_3+2c\om_3\om_4+d\om_4\om_4,\quad ad<0$, $\hspace*{1mm}r=p^2+p$& $\frac{p(p+1)}{a} u_1$ &$  \hspace*{6mm}0$&$  \hspace*{4mm}$ \cmark \\ \hline
$2.2^1$:$2$ $\vphantom{\displaystyle\frac{A^a}{A^a}}$&${2a(\om_1\om_3+\om_2\om_4)+b\om_2\om_2}$&$\frac{p^2-4}{4ap} u_4$&$  \hspace*{6mm}0$&$  \hspace*{4mm}$ \cmark \\ \hline
$2.2^1$:$3$ $\vphantom{\displaystyle\frac{A^a}{A^a}}$&${2a(\om_1\om_3+\om_2\om_4)+b\om_2\om_2}$& $x_1 u_1+\varsigma u_2+\frac{-1+2\varsigma b}{4a} u_4$ &$\hspace*{-1mm} \hspace*{1mm}[-\infty,\infty]$&$\hspace*{-3mm}\hspace*{1mm}\Leftrightarrow$ $x_1=0$\\
\hline
$2.5^1$:$4$ $\vphantom{\displaystyle\frac{A^a}{A^a}}$&${2a(\om_1\om_3+\om_2\om_4)+b\om_3\om_3}$&$\frac{2h-h^2+4P}{4a} u_1$&$  \hspace*{6mm}0$&$  \hspace*{4mm}$ \cmark \\ \hline
$3.3^1$:$1$ $\vphantom{\displaystyle\frac{A}{A^a}}$&${2a(\om_1\om_3+\om_2\om_4)+b\om_3\om_3}$&$\frac{p}{a} u_1$&$  \hspace*{6mm}0$&$  \hspace*{4mm}$ \cmark \\ \hline
\noalign{\hrule height 1pt}
\end{tabular}
\newpage
}
{\small
\smallskip
{\bf Table 5: Non-symmetric examples with ${\bf Q}$ of Segre type ${\bf [(11,2)]}$.}\nopagebreak\\[3 pt]
\begin{tabular}{!{\vrule width 1pt}p{1cm}|p{5.5cm}|p{3.8cm}|p{1.3cm}|p{1.3cm}!{\vrule width 1pt}}
\noalign{\hrule height 1pt}
Case&Invariant metric $g$&Vector field $X$&$\varsigma$&$X$ is invariant\\
\hline
$1.1^2$:$1$&$c(\om_1\om_1+\om_3\om_3)+2b\om_2\om_4+d\om_4\om_4,\hspace*{1mm}p=2$&$\frac{4c^2+b^2-2\varsigma dc^2}{8c^2 b} u_2+\frac{1}{2}\varsigma u_4$& $  \hspace*{1mm}[-\infty,\infty]$&$  \hspace*{4mm}$ \cmark \\ \hline
$1.1^2$:$2$ $\vphantom{\displaystyle\frac{A^a}{A^a}}$&
$\begin{array}{c}$ $\hspace*{-3mm}(1)\hspace*{3mm}c(\om_1\om_1+\om_3\om_3)+2b\om_2\om_4\\+d\om_4\om_4,\hspace*{3mm}p=2\\ \hspace*{-3mm} (2)\hspace*{3mm}c(\om_1\om_1+\om_3\om_3)+2b\om_2\om_4\\+d\om_4\om_4$ $\end{array}$
&$\begin{array}{c}$ $\hspace*{-3mm}(1)\hspace*{3mm}\frac{2-\varsigma d}{4 b} u_2+\frac{1}{2}\varsigma u_4\\ $ $\\ \hspace*{-5mm} (2)\hspace*{3mm}\frac{p-1}{bp}u_2+\varsigma u_3\\$ $\end{array}$ &$\begin{array}{c}$ $\hspace*{-3mm}[-\infty,\infty]\\ $ $\\ \hspace*{-3mm}0\\$ $\end{array}$ &$\begin{array}{c}
$ $\hspace*{0mm}$\cmark$\\$ $ \\ \hspace*{0mm}$ \cmark $\\$ $
\end{array}$\\ \hline
$1.4^1$:$9$ &$a(-2\om_1\om_3+\om_2\om_2)+b\om_3\om_3+2c\om_3\om_4
{-\frac{a(4r+1)}{4}\om_4\om_4},\hspace*{1mm}p=-\frac{1}{2}$&$\frac{-3+4r}{16a} u_1$&$  \hspace*{6mm}0$&$  \hspace*{4mm}$ \cmark \\ \hline
$1.4^1$:$10$ $\vphantom{\displaystyle\frac{A^a}{A}}$ &$a(-2\om_1\om_3+\om_2\om_2)+b\om_3\om_3+2c\om_3\om_4\newline+d\om_4\om_4,\quad ad>0$&$\frac{p(1+p)}{a} u_1$&$  \hspace*{6mm}0$&$  \hspace*{4mm}$ \cmark \\ \hline
$2.5^2$:$2$ $\vphantom{\displaystyle\frac{A^a}{A^a}}$&$2a\om_1\om_3+a(\om_2\om_2+\om_4\om_4)+b\om_3\om_3$&$\frac{r^2+p}{a} u_1$&$  \hspace*{6mm}0$&$  \hspace*{4mm}$ \cmark \\ \hline
$3.3^2$:$1$ $\vphantom{\displaystyle\frac{A^a}{A^a}}$&$2a\om_1\om_3+a(\om_2\om_2+\om_4\om_4)+b\om_3\om_3$&$\frac{p}{a} u_1$&$  \hspace*{6mm}0$&$  \hspace*{4mm}$ \cmark \\ \hline
\noalign{\hrule height 1pt}
\end{tabular}
}

\end{document}